\documentclass[12pt]{amsart}

\usepackage{amsmath}
\usepackage{amssymb}
\usepackage{amsthm}
\usepackage{hyperref}
\usepackage{xypic}
\usepackage{graphicx}

\usepackage{fullpage}
\usepackage[mathlines]{lineno}
\usepackage{setspace}
\usepackage{xcolor}

\theoremstyle{plain}
\newtheorem{definition}{Definition}
\newtheorem{proposition}{Proposition}

\newtheorem{lemma}{Lemma}
\newtheorem{corollary}{Corollary}

\theoremstyle{definition}
\newtheorem{example}{Example}

\theoremstyle{remark}

\DeclareMathOperator{\id}{id}

\DeclareMathOperator{\argmin}{argmin}
\def\shf{\mathcal}

\title{Aggregation sheaves for greedy modal decompositions}

\author{Michael Robinson}\thanks{SRC, Inc., and Department of Mathematics and Statistics, American University, Washington, DC 20016 (mrobinson@srcinc.com)}

\begin{document}

\begin{abstract}
This article develops a new theoretical basis for decomposing signals that are formed by the linear superposition of a finite number of modes.
Each mode depends nonlinearly upon several parameters; we seek both these parameters and the weights within the superposition.
The particular focus of this article is upon solving the decomposition problem when the number of modes is not known in advance.
A sheaf-theoretic formalism is introduced that describes all modal decomposition problems,
and it is shown that minimizing the local consistency radius within this sheaf is guaranteed to solve them.
Since the modes may or may not be well-separated, a greedy algorithm that identifies the most distinct modes first may not work reliably.
This article introduces a novel mathematical formalism, aggregation sheaves, and shows how they characterize the behavior of greedy
algorithms that attempt to solve modal decomposition problems.
\end{abstract}

\maketitle



\section{Introduction}
\label{sec-intro}

To explore the fundamental limits on modal imaging, we need an umbrella framework that encompasses all existing ones yet generalizes beyond them to the maximum extent possible. It has been recently argued by the author that sheaves are the appropriate mathematical foundation upon which to build a theory of topological filtering \cite{robinson2014topological}. 
Because sheaves and topological filters are foundational, 
they can be used to subsume all modal decomposition approaches as special cases, 
including those not yet discovered, 
and all algorithmic instantiations of them.
 
Recently, we demonstrated that truly revolutionary signal-to-noise performance can be obtained when compared to existing approaches if theoretical and implementation hurdles can be surmounted \cite{robinson2018dynamic}. 
This article attempts to reconcile these empirical findings with an appropriately general theory.
By starting generally, it attempts to lay the theoretical and methodological groundwork for solving \emph{all} modal decomposition problems.
One cannot overemphasize this fact: regardless of what physical processes or techniques are used for gathering measurements about sources -- be they classical, quantum, or otherwise -- they are subject to the same mathematical formalism as described in this article.
Given the generality of our approach, it is likely that (a) our methods will apply in many unexpected settings, (b) our methods may exactly recover existing techniques in certain settings, but (c) anecdotally (see \cite{robinson2018dynamic} for instance) our more general methods may substantially outperform existing modal decompositions when existing techniques are hampered by limiting assumptions about the signals being used.

\subsection{Historical context}
\label{sec-history}

Modal decompositions are old problems, so there are many frameworks in use today, such as Fourier bases, wavelets, and their generalizations to frames.  
Generalized Fourier decompositions, relying on the algebraic structure of the modes.  For instance, wavelets \cite{daubechies1992ten,rioul1991wavelets} rely on frames \cite{casazza2013introduction,casazza2000art}, which are overcomplete orthogonal dictionaries of modes.

Compressive sensing methods have become popular for solving decomposition problems (for instance, \cite{fornasier2015compressive,baraniuk2007compressive}).
They generally rely on the sparsity of the modes, which is to say that a given signal is composed of a few modes only. 
As a result, compressive sensing generally exploits both algebraic and geometric structure to reduce the necessary search space for a modal decomposition. 

Many of these decomposition techniques rely upon a greedy or partially-greedy algorithmic approach in practice.
Perhaps the most famous greedy algorithm for source decompositions in the CLEAN algorthim and its generalizations (see \cite{Hogbom_1974,mckinnon1990spectral,Stewart_2011} for instance, among many others).
These greedy algorithms rely on the geometric structure of modes, 
such as a single strongest peak within each mode. 
Although greedy algorithms can perform extremely well in certain settings, 
they can also fail spectacularly when their assumptions are violated.
This article begins to a provide a theoretical basis for ensuring the correct performance of these greedy algorithms.

Finally, recent quantum techniques have raised awareness of the possibility for substantially better performance.  
For instance, Tsang \cite{tsang2019resolving} showed that the quantum modal imaging could separate sources that were separated well below the Rayleigh limit.
This idea has already led to improvements in molecular microscopy \cite{mazidi2020quantifying}, 
in which the Wasserstein metric was used to provide statistical robustness.


\section{Detailed problem statement}
\label{sec-problem_statement}

Assume that one obtains a number of signals of the form
\begin{equation}
\label{eq:general_signal_model}
s(x;a_1,b_1, a_2, b_2, \dotsc, a_N, b_N) = \sum_{n=1}^N a_n \phi(x;b_n),
\end{equation}
where usually the \emph{number of sources} $N$ is fixed but unknown.
We will call $a_n$ and $b_n$ the \emph{parameters for source $n$}.
Our goal is to find $\{a_n\}$ and $\{b_n\}$ from samples of $s$ at a finite set of values for $x$, namely $x_1$, $x_2$, $\dotsc$, $x_M$.

We will call each $x_m$ a \emph{measurement} and $s(x_m)$ an \emph{observation}.

Equation \ref{eq:general_signal_model} is meant to be taken fairly generally.  
We do not require that $\phi$ be nonnegative, complex valued, or real valued.
(For simplicity of exposition \emph{only}, we will assume that $\phi$ is real valued.
The reader may generalize this \emph{mutatis mutandis} if desired.) 
That said, we will assume that $\phi(x,y)$ is smooth in the second variable $y$,
but we need not assume that it is linear in either variable.  
If $\phi(x,y)$ is not continuous in the first variable $x$, 
this may limit certain strategies for adaptively determining the next measurement.
However, even if $\phi(x,y)$ is not continuous in the first variable $x$,
all of the theory and some of the methodology will still work without changes.

What follows are two separate source parameter recovery problems that are special cases of our assumed signal model.
When we solve the problem of recovering source parameters in the general case in Section \ref{sec-approach}, 
our solution \emph{automatically applies} to each of the special cases described in this section.  
Although there are algorithmic considerations, 
the theory provides a surprisingly complete set of solvability conditions (Proposition \ref{prop:whitney}) and prescribes a rather definite algorithmic recipe.
Quite specifically, the algorithms that solve each of the problems in this section amount to special cases of solvers for the general optimization problems posited in the statements of Corollary \ref{cor:known_source_optimization} and Proposition \ref{prop:unknown_source_optimization}.
These optimization problems are \emph{consistency radius minimization problems} whose use has been recently effective in solving many complex signal processing problems \cite{robinson2019hunting, robinson2018dynamic}.

\subsection{Fourier decomposition}

In the most basic case, Equation \ref{eq:general_signal_model} represents the Fourier series decomposition of a periodic function.
For this representation, we merely need to take $b_n = n \in \mathbb{Z}$ and
\begin{equation*}
\phi(x;n) = e^{2\pi i n x},
\end{equation*}
where $n \in \mathbb{Z}$, which is to say that $n$ is an arbitrary integer.

The signal model is then simply
\begin{equation*}
s(x) = \sum_{n=-N}^N a_n e^{2\pi i n x}.
\end{equation*}
Since we are only concerned with finitely many sources, we will not consider what happens if $N\to \infty$.

\subsection{Spectral Estimation}
\label{sec-spectral_estimation}

We can generalize the Fourier series decomposition to handle non-integer frequencies by a slight change to the signal model,
\begin{equation*}
\phi(x;\omega) = e^{2\pi i \omega x},
\end{equation*}
where $\omega \in \mathbb{R}$.  
The signal model is
\begin{equation*}
s(x) = \sum_{n=1}^N a_n e^{2\pi i \omega_n x},
\end{equation*}
so that the complex magnitudes $\{a_n\}$ and the real frequencies $\{\omega_n\}$ are parameters.  
Of course this signal model really isn't much different from the Fourier decomposition, 
but it differs in one important aspect: \emph{convergence}.  
With finitely many sources (ie. $N < \infty$), 
this decomposition is a finite sum even though its Fourier decomposition may require infinitely many terms.  
Worse, the Fourier series usually only converges in the $L^2$ sense (not pointwise) and may do so slowly.


\section{Detailed technical approach}
\label{sec-approach}

This section details our technical approach in detail. 
The story is rather intricate,
as the demands of the problem couple with the demands of adaptive algorithms.
Right from the start, the general formulation of the problem in Section \ref{sec-problem_statement}
beckons for one to discover its proper mathematical foundations.

Consider that each source is determined by its magnitude and location.
These properties have different units and thereby suggest that the correct \emph{topological space} for sources is not Euclidean space.
As explained in Section \ref{sec-source_parameters}, we have discovered that the correct space for sources is a quotient of a product of \emph{conical} spaces.
Interpreting a set of source parameters as a point within such a quotient space, 
we can connect measurements of the sources (how ever they are obtained) to the coordinates of this point.
Differential topology provides a clear, concrete answer as to how many measurements must be obtained in order to succeed at this task,
which we outline in Section \ref{sec-free_parameters}.

While a non-adaptive approach for source decomposition has a straightforward incarnation as a sheaf we call $\shf{J}_P$ (as we describe in Sections \ref{sec-sheaf_known} and \ref{sec-sheaf_unknown}),
adaptive approaches require a novel mathematical tool called an \emph{aggregation sheaf}.
Since aggregation sheaves are a completely new discovery, we explore their properties in Section \ref{sec-aggregation}.
Finally, we apply aggregation sheaves to support an adaptive approach to general source decomposition in Section \ref{sec-sheaf_nested}.
The result is a sheaf we call $\shf{K}_P$.

\subsection{Source parameter space representation}
\label{sec-source_parameters}

The problem at hand is determining the parameters of the sources, namely the values of $a_n$ and $b_n$ in Equation \eqref{eq:general_signal_model} 
from the values of $s(x_1)$, $s(x_2)$, $\dotsc$, $s(x_M)$.  

To fix ideas, let us first consider each source individually.
It happens that this is already an interesting -- and apparently underappreciated -- problem.
Assume that $a_n \in \mathbb{R}^A$ and $b_n \in \mathbb{R}^B$ for each $n$, for some fixed $A$, $B$.
Conventionally, we call $a_n$ the \emph{magnitude} of the source and $b_n$ its \emph{location}.

It is convenient to encapsulate the source parameters into a single manifold $I$,
whose points consist of pairs $w=(a,b) \in \mathbb{R}^A\times \mathbb{R}^B \cong \mathbb{R}^{A+B}$
under the identification that $(0,b) = (0,b')$ for any $b$ and $b'$.
The identification captures the intuition that if a source has no magnitude, the other parameters do not matter.
We therefore call any element $(0,b)$ to be a \emph{zero element} of $I$.

The manifold $I$ should have a metric that reflects the zero elements.
We can enforce this requirement if we assume that $a_i \in \mathbb{R}^A$ and $b_i \in \mathbb{R}^B$
\begin{equation}
\label{eq:I_metric}
d_I((a_1,b_1),(a_2,b_2)) := \min\left\{\|a_1\|+\|a_2\|,\sqrt{\|a_1-a_2\|^2 + \beta \|b_1 - b_2\|^2}\right\},
\end{equation}
where $\beta>0$ is a constant and the usual norms on $\mathbb{R}^A$ and $\mathbb{R}^B$ are assumed.
This metric is called a \emph{bottleneck metric}, and is widely used in applications of topology.
Notice in particular that this reduces to the case of $d_I((0,b),(0,b'))=0$ and $d_I((a,b),(a',b)) = \|a-a'\|$.

\begin{lemma}
The bottleneck metric is a pseudometric on $\mathbb{R}^{A+B}$, and is a metric on $I$.
\end{lemma}
\begin{proof}
We simply need to establish the usual properties of a pseudometric:
\begin{description}
\item[Reflexivity] $d_I((a,b),(a,b)) = \min\left\{\|a\|+\|a\|,\sqrt{\|a-a\|^2 + \beta \|b - b\|^2}\right\} = \min\{2\|a\|,0\}=0$.  
\item[Symmetry] $d_I((a_1,b_1),(a_2,b_2)) = d_I((a_2,b_2),(a_1,b_1))$ because the defining Equation \ref{eq:I_metric} is clearly symmetric.
\item[Triangle inequality] The usual trick of adding and subtracting, along with the usual triangle inequality, establishes that
\begin{equation*}
\begin{aligned}
\sqrt{\|a_1-a_3\|^2 + \beta \|b_1 - b_3\|^2} & =  \sqrt{\|a_1-a_3 + a_2 - a_2\|^2 + \beta \|b_1 - b_3 + b_2 - b_2\|^2}\\
& \le  \sqrt{\|a_1-a_2\|^2 + \beta \|b_1 - b_2\|^2 } +\sqrt{\| a_2 - a_3\|^2 +\beta\| b_2 - b_3\|^2}\\
\end{aligned}
\end{equation*} 
and it is clear that $\|a_1\|+\|a_3\| \le \|a_1\| + 2\|a_2\| + \|a_3\|$.  
Combining these facts,
\begin{equation*}
\begin{aligned}
d_I((a_1,b_1),(a_3,b_3))  =& \min\left\{\|a_1\|+\|a_3\|,\sqrt{\|a_1-a_3\|^2 + \beta \|b_1 - b_3\|^2}\right\}\\
\le& \min\left\{\|a_1\|+2\|a_2\|+\|a_3\|,\sqrt{\|a_1-a_2\|^2 + \beta \|b_1 - b_2\|^2 } \right.\\&\left.+\sqrt{\| a_2 - a_3\|^2 +\beta\| b_2 - b_3\|^2}\right\}\\
\le& \min\left\{\|a_1\|+\|a_2\|,\sqrt{\|a_1-a_2\|^2 + \beta \|b_1 - b_2\|^2 } \right\} \\&+ 
\min\left\{\|a_2\|+\|a_3\|,\sqrt{\| a_2 - a_3\|^2 +\beta\| b_2 - b_3\|^2}\right\}\\
\le& d_I((a_1,b_1),(a_2,b_2))  + d_I((a_2,b_2),(a_3,b_3)). 
\end{aligned}
\end{equation*}
\end{description}
Finally, if $d_I((a_1,b_1),(a_2,b_2)) = 0$, this means that either $\|a_1\|+\|a_2\|=0$ or $\|a_1-a_2\|^2 + \beta \|b_1 - b_2\|^2=0$.
In the latter case, it is clear that $a_1=a_2$ and $b_1=b_2$, since $\beta>0$ and the usual norms induce metrics on Euclidean space.
In the former case, we conclude that $a_1=a_2=0$, which means that both points correspond to the vertex of $I$.
These two cases therefore establish that $d_I$ is a metric on $I$.
\end{proof}

\begin{definition}
\label{def:conical_space}
A \emph{conical space} is a metric space $I$ that can be written as the quotient of a product $\mathbb{R}^A \times \mathbb{R}^B$ with a metric $d_I$ of the form given in Equation \ref{eq:I_metric}, where the quotient identifies sets of points that are distance zero apart from each other.
\end{definition}

As a shorthand, we will use the notation
\begin{equation*}
\|(a,b)\|_I := d_I((0,c),(a,b)) =\min\left\{\|a\|,\sqrt{\|a\|^2 + \beta \|c - b\|^2}\right\} = \|a\|,
\end{equation*}
which works a bit like a norm on $I$ because it recovers the norm on $\mathbb{R}^A$.
We will call $\|(a,b)\|_I$ the \emph{magnitude norm} of $(a,b)$ in $I$.

\begin{lemma}
The metric $d_I$ on $I$ satisfies the following properties:
\begin{enumerate}
\item If $s$ is a real number, $\|(sa,sb)\|_I = |s| \|(a,b)\|_I$,
\item $\|(a,b)\|_I=0$ if and only if $a=0$, and
\item $\|(a_1+a_2,b)\|_I \le \|(a_1,b)\|_I + \|(a_2,b)\|_I$ for all $a_1$, $a_2$, and $b$.
\end{enumerate}
\end{lemma}
\begin{proof}
Each of these statements are immediate from the definition of $\|(a,b)\|_I$.
\end{proof}

\begin{corollary}
\label{cor:normed_is_conical}
A normed space $V$ is a conical space of the form $V=V\times\{0\}$, namely the product of $V$ with the trivial vector space.
\end{corollary}

\subsection{Free parameter analysis}
\label{sec-free_parameters}

Let us start with the situation where the number of sources $N$ is known,
before we turn to the more general setting in Section \ref{sec-sheaf_unknown}.

Let us assume that measurements are given by a function $s$ taking values $s(x_m) \in \mathbb{R}$. 
(These assumptions are not terribly critical in what follows, but it makes explaining the dimension counting easier.
The reader can easily adjust the dimensions if desired.)

It is most useful to bundle all of the measurements into a vector $(s(x_1), s(x_2), \dotsc, s(x_M)) \in \mathbb{R}^M$.  
Under the usual topology on $\mathbb{R}^M$, this means that
\begin{equation*}
s(x) = \sum_{n=1}^N a_n \phi(x;b_n)
\end{equation*}
induces a smooth function $S: I^N = \mathbb{R}^{N(A+B)} \to \mathbb{R}^M$ from the $N$-fold product of the conical space $I$, given by
\begin{equation}
\label{eq:S_function}
S(a_1,b_1, a_2, b_2, \dotsc, a_N,b_N) := (s(x_1), s(x_2), \dotsc, s(x_M)).
\end{equation}  

Notice that Equation \ref{eq:S_function} is invariant to permutations of the sources.  
That is,
\begin{equation*}
S(a_1,b_1, \dotsc, a_i, b_i, \dotsc, a_j, b_j \dotsc, a_N,b_N) = S(a_1,b_1, \dotsc, a_j, b_j, \dotsc, a_i, b_i \dotsc, a_N,b_N)
\end{equation*}
for all pairs of integers $i,j = 1, \dotsc, N$.  
Put somewhat more abstractly, if $S_N$ is \emph{$N$-fold symmetric group}, the group of permutations of $N$ items,
then $S$ factors through $I^N/S_N$.
This statement is equivalent to the commutative diagram
\begin{equation*}
\xymatrix{
I^N \ar[r]^S \ar[d] & \mathbb{R}^M\\
I^N/S_N \ar[ur]_{S'}
}
\end{equation*}
In the rest of this article, we will abuse notation by allowing $S$ to refer either to $S:I^N \to \mathbb{R}^M$ or $S': I^N/S_N \to \mathbb{R}^M$
in the diagram above.
Since both $I^N$ and $I^N/S_N$ are manifolds of the same dimension, 
and both maps are defined by the same formula (Equation \ref{eq:S_function}),
this should cause little confusion.  

The problem of determining the parameters of the sources can be solved in principle if and only if $S$ is injective.  
Clearly a necessary condition for this to occur is that $M \ge N(A+B)$.

Under fairly general conditions, $S$ is usually injective if $M$ is large enough.

\begin{proposition} (the Whitney embedding theorem \cite{lee2013smooth}; see also \cite{robinson2012topological})
\label{prop:whitney}
Suppose that $\epsilon>0$ is given.
If $M > 2N(A+B)$, then there is an injective smooth function $I^N\to \mathbb{R}^M$ from source parameters to measurements that is within a distance of $\epsilon$ to $S$ in the $C^\infty$ topology of functions on $I^N$.
\end{proposition}
  
Intuitively, when $M > 2N(A+B)$, the only way that $S$ can fail to be injective is if there is a symmetry present in the source parameters and the samples.
This means that in principle, the sources can be completely inferred from the set of samples.

It is important to notice that the injectivity of $S$ ensures that the localization problem has a unique solution, 
irrespective of other properties of the signal model, such as its wavelength.
Therefore, any algorithm that can invert $S$ along its image is not subject to resolution limits, 
though it may still incur discretization errors or truncation errors if it is iterative.
Nevertheless, any such algorithm that converges to a true inverse of $S$ can be said to solve a general super-resolution problem.

Although the proof the Whitney embedding theorem is not constructive, 
it is typically proven using a perturbation-based argument.
These perturbations may be rather undesirable within the context of the signal model, 
for instance they might suggest adding a grating lobe to an antenna pattern.
More sensible perturbations can be obtained using a weaker form of the Whitney embedding theorem, 
called the signal embedding theorem \cite{robinson2012topological},
which allows the support of $\phi(\cdot;b_n)$ to be compact.
Perturbations need not occur outside these supports.

The signal embedding theorem has been applied to invert signal models in the same form as $S$ for source localization \cite{robinson2012topological} and antenna measurement \cite{robinson2014knowledge}.
The perturbation arguments suggest that an adaptive algorithm for inverting $S$ along its image might exist.
Nevertheless, the basic problem is that we need an algorithm for this inversion.  
Section \ref{sec-sheaf_known} uses sheaves to recast the source recovery problem into an optimization problem that can be solved effectively by an adaptive approach.  

\subsection{Sheaf signal model: known source count}
\label{sec-sheaf_known}

If we know the number of sources $N$, we can express the relationship between the observations and the source parameters by a sheaf diagram
\begin{equation}
\label{eq:fixed_source_sheaf}
\xymatrix{
&I^N/S_N \ar[dl]_{s(x_1;\cdot)} \ar[d]^{s(x_2;\cdot)} \ar[drr]^{s(x_M;\cdot)} \\\
\mathbb{R} & \mathbb{R} & \dotsb & \mathbb{R}
}
\end{equation}
The diagram expresses the fact that each observation is determined by Equation \ref{eq:general_signal_model}, 
and that no observation functionally determines any other.
One should be aware that the lack of functional dependence does not imply that the observations are independent.
Specifically, it may happen that one or more measurements (values in the second row of the diagram) suffice to determine a unique value in $I^N/S_N$,
after which point all of the other measurements can be determined from this unique value.
This is exactly the situation in which an adaptive method is warranted, since later measurements are not actually required!

Suppose we have a set of $M$ measurements at $x_1$, $x_2$, $\dotsc$, $x_M$, 
for which we have observations $z_1$, $z_2$, $\dotsc$, $z_M$ respectively.
These observations form a partial assignment to the above sheaf, which is supported on the second row.

The partial assignment is not supported at the top row, since that would be tantamount to knowing all of the source parameters.
Nevertheless, assuming that no noise is present and that we have correctly found the parameters $a_1$, $b_1$, $\dotsc$, $a_N$, $b_N$, 
then it should be the case that
\begin{equation*}
z_m = s(x_m;a_1,b_1,\dotsc,a_N,b_N)
\end{equation*}
for all $m=1, \dotsc, M$.  

In other words, the set of observations corresponds to a global section of the sheaf diagram in Equation \ref{eq:fixed_source_sheaf}.
Since global sections have zero consistency radius and the stalks of the sheaf are all metric spaces, 
any assignment which is not a global section will have positive consistency radius.

On the other hand, if the function $S$ defined in Equation \ref{eq:S_function} is injective then there is only one global section.
This means that we can recast the problem of obtaining the sources from the observations as the minimization of consistency radius.
With essentially no further work, we obtain the following result.

\begin{corollary}
\label{cor:known_source_optimization}
If the function $S$ defined in Equation \ref{eq:S_function} is injective then the solution to
\begin{equation}
\label{eq:known_source_optimization}
\argmin_{\{a_n,b_n\}_{n=1}^N} \left(\sum_{m=1}^M \left|s(x_m;a_1,b_1, \dotsc, a_n,b_n) - z_m \right|^2\right).
\end{equation}
is the correct source decomposition for the signal $z_m$ if one exists.
\end{corollary}

The reader may correctly argue that the optimization problem in Equation \ref{eq:known_source_optimization} can be obtained easily -- though not solved -- by inspection!
However, if the number of sources is not known, 
then the correct optimization problem that obtains the source parameters is difficult to state explicitly.
However, the correct optimization problem is still a consistency radius minimization,
but for a more elaborate sheaf.
This sheaf accounts both for the unknown number of sources and the fact that (according to Section \ref{sec-free_parameters}) 
the number of measurements that should be taken depends on the (unknown) number of sources.

\subsection{Sheaf signal model encore: unknown source count}
\label{sec-sheaf_unknown}

Now let us assume that the signal is built from an unknown number of sources.
Recalling that we defined the conical space $I$ so it can represent the space of possible pairs of values for $a_n$ and $b_n$,
this means that a single source is specified by an element of the set $I$.
On the other hand, we previously argued that if there are $N$ sources, 
then the order of the sources does not matter.
Therefore, the signal for $N$ sources is determined by an element of $I^N/S_N$.
 
If $M$ real measurements are taken from a signal determined by $N$ sources,
then the previous discussions led to a simple sheaf diagram
\begin{equation*}
\xymatrix{
I^N/S_N \ar[d]\\ \mathbb{R}^M
}
\end{equation*}
The consistency radius is simply
\begin{equation*}
\|S(w_{1}, \dotsc, w_{N}) - z\|,
\end{equation*}
where $z \in \mathbb{R}^M$ is the measurement and $w_i=(a_i,b_i)$ is the parameter of the $i$-th source.
If we do not know the number of sources, then we must consider various instantiations of this sheaf.

\begin{definition}
\label{def:j_p}
The sheaf $\shf{J}_P$ is defined by the diagram
\begin{equation}
\label{eq:j_p}
\xymatrix{
I\ar[drr]_{S(\cdot)} && I^2/S_2\ar[d]_{S(\cdot,\cdot)} && I^3/S_3 \ar[dll]_{S(\cdot,\cdot,\cdot)} && \dotsb & I^P/S_P \ar[dlllll] \\
&&\mathbb{R}^M\\
}
\end{equation}
in which each of the restriction maps are given by the signal space map with different numbers of sources.
\end{definition}

Notice that the $P$ in $\shf{J}_P$ is the maximum number of sources we will attempt to consider, 
which may not have much (if anything) to do with the actual number of sources $N$.
Of course, we hope that $N \le P$ so that the actual number of sources is considered in our analysis.
The global consistency radius of an assignment to this sheaf is given by
\begin{equation*}
\sum_{i=1}^P \|S(w_{i,1}, \dotsc, w_{i,i}) - z\|,
\end{equation*}
where $z \in \mathbb{R}^M$ is the measurement and the $w_{i,j}$ is the proposed $j$-th source parameter when we are considering a representation of the signal with $i$ sources.
The global consistency radius is simply the sum of consistency radii of each subproblem, in which we consider a given number of sources.
On the other hand, each one of these subproblems is the \emph{local} consistency radius for an open set in the base space topology for $\shf{J}_P$.
If the $S$ functions are injective, then it is clear that if $N>P$, the only set upon which the local consistency radius can vanish is the set consisting of only the observation (ie. the bottom row of the sheaf diagram).
Conversely, if the $S$ functions are injective and $N \le P$, then the minimum local consistency radius possible on the open set constructed as the star over the element with stalk $I^P/S_P$ is zero.

One particular drawback of $\shf{J}_P$ though, is that values of an assignment on the different stalks in the top row of Equation \ref{eq:j_p} have nothing to do with one another.
Another expression of this fact is that the base space topology for $\shf{J}_P$ is quite large.

\subsection{Aggregation sheaves}
\label{sec-aggregation}

It is useful to construct a specific kind of sheaf that aggregates observations or parameter values under different assumptions.
In this section, we define the concept of an \emph{aggregation sheaf} that represents the idea of nested sets of parameters.

Suppose that $I$ is a conical space and $S_N$ is the group of permutations of $N$ items.
The elements of $I^N/S_N$ are equivalence classes of elements of $I^N$ under permutations.
We may choose representatives of these equivalence classes so that the components (elements of $I$) are sorted according to their magnitude norms.

For instance, a typical representative of an element of $I^N/S_N$ could be written
\begin{equation*}
\{w_1, w_2, \dotsc, w_N\}
\end{equation*}
where $\|w_i\|_I \ge \|w_j\|_I$ if $i \le j$.

A natural operation is \emph{concatenation} $\oplus: (I^M/S_M)\times(I^N/S_N) \to I^{M+N}/S_{M+N}$ given by
\begin{equation*}
\{u_1,\dotsc,u_M\} \oplus \{w_1, \dotsc,w_N\} := \{u_1, \dotsc, u_M, w_1, \dotsc, w_N\}.
\end{equation*}
We may therefore abuse notation slightly, and think of elements of $I^N/S_N$ as length $N$ \emph{multisets}: 
unordered lists of length $N$ in which duplicates are permitted.
We will therefore speak of $w_1$ as being an \emph{element of} $\{w_1,\dotsc,w_M\}$ without any possible confusion.

For a conical space $I$, metric $d_I$ induces a natural metric on $I^N/S_N$.

\begin{lemma}
\label{lem:conical_metric}
If $I$ is a conical space, then $I^N/S_N$ is a metric space in which the metric is given by
\begin{equation*}
d(u,w):= \min_{\sigma \in S_N} \sum_{n=1}^N d_I\left(u_n, w_{\sigma(n)}\right).
\end{equation*}
\end{lemma}
\begin{proof}
We simply need to establish the axioms of a metric:
\begin{itemize}
\item Because $d(u,w)$ is a sum of norms, $d(u,w) \ge 0$ is immediate.
\item Similarly, $d(w,w) = 0$ is immediately apparent by taking $\sigma = \id$.
\item If $d(u,w) = 0$, this means that there is a permutation $\sigma$ such that $u_n = w_{\sigma(n)}$ for each $n$ because $d_I$ is a metric.  This implies that $u$ and $w$ are in the same equivalence class of $I^N/S_N$.
\item A simple reindexing argument establishes that $d(u,w) = d(w, u)$.
\item Finally, $d(u,v) \le d(u,w) + d(w,v)$
\begin{eqnarray*}
d(u,v) &=&\min_{\sigma \in S_N} \sum_{n=1}^N d_I\left(u_n, v_{\sigma(n)}\right)\\
&\le&\min_{\tau \in S_N} \min_{\sigma \in S_N} \sum_{n=1}^N \left( d_I\left(u_n , w_{\tau(n)}\right) + d_I\left(w_{\tau(n)}, v_{\sigma(n)}\right)\right)\\
&\le&\min_{\tau \in S_N} \min_{\sigma \in S_N} \left(\sum_{n=1}^N d_I\left(u_n, w_{\tau(n)}\right) + \sum_{n=1}^Nd_I\left(w_{\tau(n)}, v_{\sigma(n)}\right)\right)\\
&\le&\min_{\tau \in S_N} \sum_{n=1}^N d_I\left(u_n, w_{\tau(n)}\right) + \min_{\tau \in S_N} \min_{\sigma \in S_N}\sum_{n=1}^Nd_I\left(w_{\tau(n)}, v_{\sigma(n)}\right)\\
&\le&d(u,w) + d(w,v).
\end{eqnarray*}
\end{itemize}	
\end{proof}

With these basic properties established, we now aggregate copies of $I$ according to a recipe encoded as a partially ordered set $(P,\le)$.
Intuitively, the elements of $I$ represent choices of parameters for the sources, 
and the elements of $P$ represent different possible ways to configure these sources.
The partial order $\le$ on $P$ represent nesting relations between these configurations.

For instance the relation between a configuration with $2$ sources can be thought of as a sub-configuration of one with $3$ sources in several possible ways.
The appropriate representation of the space of parameters along with these relationships is an \emph{aggregation sheaf}.

\begin{definition}
\label{def:aggregation_sheaf}
Suppose that $D_p$ is the subset of $P$ in the partially ordered set $(P,\le)$ given by
\begin{equation*}
D_p := \{q \in P : q \le p \},
\end{equation*}
and that $I$ is a conical space.
The \emph{aggregation sheaf $\shf{A}$ on $(P,\le)$ modeled on $I$} is a sheaf of metric spaces constructed according to the following recipe:
\begin{description}
\item[Stalks] $\shf{A}(p) := \left(\bigoplus_{q\in D_p} I\right)/S_{\#D_p}$
\item[Restrictions] If $\shf{S}(q)$ is a multiset of $N$ elements $\{w_1, \dotsc, w_N\}$ and $\shf{S}(p)$ is a multiset of $N+k$ elements, then
\begin{equation*}
\left(\shf{S}(q\le p)\right)\left(\{w_1, \dotsc, w_N\}\right) := \{w_1, \dotsc, w_N, (0,b_1), \dotsc, (0,b_k)\},
\end{equation*}
where $b_1, \dotsc, b_k$ are completely arbitrary, since they all refer to the vertex of $I$.  Briefly, $\shf{S}(q\le p)$ copies the elements of $\shf{S}(q)$, padding with copies of the vertex of $I$ as needed.  We call $\shf{S}(q\le p)$ the \emph{inclusion} of $\shf{S}(q)$ into $\shf{S}(p)$.
\end{description}
\end{definition}
According to Lemma \ref{lem:conical_metric}, the stalks of an aggregation sheaf are metric spaces.  
It is immediate that the restrictions are continuous with respect to the metrics on the stalks.

For our purposes here -- decomposing signals formed according to Equation \ref{eq:general_signal_model} -- we want to understand the most efficient way to solve certain consistency radius minimization problems.  
Constraining the assignments that must be tested as the optimization problem is solved can yield a performance boost, 
with greedy algorithms being preferred in the literature due to their speed of convergence.
However, greedy algorithms can fail spectacularly if their assumptions are not met.
We seek both the correct optimization problem to decompose Equation \ref{eq:general_signal_model} 
and a flexible algorithmic framework that becomes greedy when appropriate, but fails more gracefully into a non-greedy setting if needed.
This graceful degradation is governed by properties of the assignments which minimize consistency radius;
a key starting point is a sorting property for assignments of aggregation sheaves.

\begin{proposition} (Sorting property)
\label{prop:sorting}
Suppose that $\shf{A}$ is an aggregation sheaf for a finite partially ordered set $(P,\le)$ modeled on the conical space $I$, 
that $(P,\le)$ has a unique maximal element $p'$, and that $N = \# P$.
Let $a$ be an assignment to $\shf{A}$ supported on $p'$.  
Without loss of generality, suppose that $a_{p'} = \{w_1, w_2, \dotsc x_N\}$, where $\|w_i\|_I \ge \|w_j\|_I$ if $i \le j$,
recalling that $\|\cdot\|_I$ is the magnitude norm of $I$.
Any extension $b$ of $a$ to all of $\shf{A}$ which has minimal consistency radius will have the property that if $q \le p$ as elements in $P$,
\begin{equation*}
b_p = b_q \oplus \{w_{i_1},\dotsc\},
\end{equation*}
where every element in the set $w_{i_1},\dotsc$ has magnitude norm less than or equal to that of every element in $b_q$, 
and each $w_{i_j}$ is an element of $a_{p'}$.
\end{proposition}
\begin{proof}
It suffices to establish the conclusion for an arbitrary pair $q \le p$ in $P$ for which there is no other $r$ in $P$ ``between'' $q$ and $p$, 
that is,  $q \le r \le p$ implies $r=q$ or $r=p$.
Under this hypothesis, the lengths of $b_p$ and $b_q$ differ by exactly $1$.
The sheaf diagram in the vicinity of $p$ and $q$ is therefore of the form
\begin{equation*}
\xymatrix{
I^N/S_N \ar[rr]^-{\shf{A}(q \le p)} && I^{N+1}/S_{N+1}.
}
\end{equation*}  
Without loss of generality, suppose that $b_p = \{w_1, w_2, \dotsc, w_{N+1}\}$ in which the elements are sorted in order of descending magnitude norms.
The proposition will be established if we can conclude that $b_q = \{w_1, w_2, \dotsc, w_N\}$.

Suppose that $b_q = \{u_1, u_2, \dotsc u_N\}$, so that
\begin{eqnarray*}
d\left(\left(\shf{A}(q\le p)\right) b_q, b_p\right) &=& \min_{\sigma\in S_{N+1}} \sum_{n=1}^{N+1} d_I\left( \left(\left(\shf{A}(q\le p)\right) b_q\right)_n, w_{\sigma(n)}\right)\\
&=&\min_{\sigma\in S_{N+1}} \left( \sum_{n=1}^{N} d_I\left( u_n, w_{\sigma(n)}\right) + d_I\left((0,b),w_{\sigma(N+1)}\right)\right)\\
&=&\min_{\sigma\in S_{N+1}} \left( \sum_{n=1}^{N} d_I\left( u_n, w_{\sigma(n)}\right) + \|w_{\sigma(N+1)}\|_I\right).
\end{eqnarray*} 
Notice that since $\shf{A}(q \le p)$ is an inclusion, its output contains one extra copy of the vertex of $I$ as padding.  
This extra vertex copy is paired with $w_{\sigma(N+1)}$, resulting in the magnitude norm in the above calculation.
Since we sorted the $w_n$, selecting $u_n = w_n$ for $n=1, \dotsc, N$ means that the above distance is not more than $\|w_{N+1}\|_I$.
On the other hand, notice that
\begin{equation*}
\min_{\sigma\in S_{N+1}} \left( \sum_{n=1}^{N} d_I\left( u_n, w_{\sigma(n)}\right) + \|w_{\sigma(N+1)}\|_I\right) \ge \min \{\|w_n\|_I : n = 1, \dotsc, N+1\} = \|w_{N+1}\|_I.
\end{equation*}
Therefore, selecting $u_n = w_n$ for $n=1, \dotsc, N$ results in the minimum value of $d\left(\left(\shf{A}(q\le p)\right) b_q, b_p\right)$.
Since this is the contribution of $b_q$ to the consistency radius of the entire assignment $b$, we conclude that it minimizes consistency radius.
\end{proof}

According to Proposition \ref{prop:sorting}, the assignments with minimal consistency radius on the aggregation sheaf
\begin{equation*}
\xymatrix{
I \ar[r] & I^2/S_2 \ar[r] & \dotsb \ar[r] & I^N/S_N
}
\end{equation*}
are always of the form
\begin{equation*}
\xymatrix{
\{w_1\} \ar[r] & \{w_1,w_2\} \ar[r] & \dotsb \ar[r] & \{w_1,\dotsc, w_N\}
}
\end{equation*}
where $\|w_i\|_I \ge \|w_j\|_I$ if $i \le j$.  
That is, the elements of $I$ added as one moves rightward in the diagram are sorted in descending order.

\subsection{Sheaf model for nested subproblems}
\label{sec-sheaf_nested}

To remedy the issues with $\shf{J}_P$, and to create a sheaf model suitable for greedy algorithms, 
let us repackage the top row of Equation \ref{eq:j_p} using an aggregation sheaf.
Since we do not know the true number of sources, it makes sense to encapsulate all possibilities in a sheaf $\shf{I}$.
This sheaf is an aggregation sheaf
\begin{equation*}
\xymatrix{
I \ar[r] & I^2/S_2 \ar[r] & I^3/S_3 \ar[r] & \dotsb \ar[r] & I^n/S_n \ar[r] & \dotsb
}
\end{equation*}
which represents all possible sets of sources present in the signal along with the nesting structure.  
At any place in the diagram above, 
there are a finite number of sources, 
and these sources are present in all positions to the right.  

This sheaf $\shf{I}$ can be connected to the measurements by replicating the diagram in Equation \ref{eq:fixed_source_sheaf}, 
with the added constraint that duplicate observations are connected by identity maps.
For two observations, this results in the diagram
\begin{equation*}
\xymatrix{
I \ar[rr]\ar[d]_{s(x_1;\cdot)} \ar[ddr]_(0.65){s(x_2;\cdot)} && I^2/S_2\ar[rr]\ar[d]_{s(x_1;\cdot,\cdot)} \ar[ddr]_(0.65){s(x_2;\cdot,\cdot)} && I^3/S_3 \ar[rr]\ar[d]_{s(x_1;\cdot,\cdot,\cdot)} \ar[ddr]_(0.65){s(x_2;\cdot,\cdot,\cdot)} && \dotsb \\
\mathbb{R} \ar[rr]^{\id} && \mathbb{R} \ar[rr]^{\id} && \mathbb{R} \ar[rr]^{\id} && \dotsb \\ 
&\mathbb{R} \ar[rr]^{\id} && \mathbb{R} \ar[rr]^{\id} && \mathbb{R} \ar[rr]^{\id} && \dotsb \\
}
\end{equation*}
Additional observations can be supplied by adding more rows to the diagram.

If we collapse the identity maps in the bottom rows, and combine the bottom rows, the above diagram has the same global sections as the sheaf $\shf{K}_P$ defined below.

\begin{definition}
\label{def:k_p}
The sheaf $\shf{K}_P$ is defined by the diagram
\begin{equation}
\label{eq:k_p}
\xymatrix{
I \ar[rr]\ar[drr]_{S(\cdot)} && I^2/S_2\ar[rr]\ar[d]_{S(\cdot,\cdot)} && I^3/S_3 \ar[rr]\ar[dll]_{S(\cdot,\cdot,\cdot)} && \dotsb & I^P/S_P \ar[dlllll] \\
&&\mathbb{R}^M\\
}
\end{equation}
in which the top row is an aggregation sheaf, and the restrictions from the top to bottom row are given by the signal map $S$.
\end{definition}

It should be immediately apparent to the reader that the consistency radius for this new sheaf $\shf{K}_P$ is rather formidable to write explicitly,
although it is straightforward to construct systematically. 
Specifically, because of commutativity all restriction maps from the top to bottom row are actually visible in the diagram.
Moreover, the top row is an aggregation sheaf, so its consistency radius can be written explicitly using the construction in Section \ref{sec-aggregation}.

Consider an assignment to the above sheaf $\shf{K}_P$, in which the values in the stalks in the top row are given by
\begin{equation*}
\{w_{1,1}\}, \{w_{2,1}, w_{2,2}\}, \{w_{3,1}, w_{3,2}, w_{3,3}\}, \dotsc
\end{equation*}
and the value in the bottom row (the $\mathbb{R}^2$) is $z$.
The consistency radius of this assignment is
\begin{equation}
\label{eq:consistency_radius_fixed_measurements}
\sum_{i=1}^N \sum_{j=i+1}^N \min_{\sigma \in S_{j}}\left( \sum_{k=1}^{i} d_I\left(w_{i,k}, w_{j,\sigma(k)}\right) + \sum_{m=i+1}^{j}\|w_{j,\sigma(m)}\|_I \right) + \sum_{i=1}^N \|S(w_{i,1}, \dotsc, w_{i,i}) - z\|.
\end{equation}

Let us compare $\shf{K}_P$ with the sheaf $\shf{J}_P$ defined earlier in Equation \ref{eq:j_p}.  
Although the base spaces of these two sheaves have different topologies, 
they are written on the same set of elements.
The only difference is the partial order on these elements.
In particular, every order relation present in the base of $\shf{J}_P$ (arrows from the top to the bottom row) is present in the base of $\shf{K}_P$,
though not conversely. 
This means that the identity map from the base space of $\shf{J}_P$ to the base space of $\shf{K}_P$ is order preserving.
Furthermore, notice that the stalks and restrictions in $\shf{J}_P$ are present in exactly the same form as in $\shf{K}_P$,
though again not conversely.
Therefore, there is a sheaf morphism $m: \shf{K}_P \to \shf{J}_P$ defined along this identity map, 
simply by defining each component map to be an identity map on corresponding stalks.

\begin{proposition}
\label{prop:k_p_j_p}
The sheaf morphism $m: \shf{K}_P \to \shf{J}_P$ induces an isomorphism on the spaces of global sections provided that $S : I^P/S_P \to \mathbb{R}^M$ is injective.
\end{proposition}
\begin{proof}
The fact that $m$ takes each section of $\shf{K}_P$ to a section of $\shf{J}_P$ is a general fact for sheaf morphisms 
and is true regardless of the injectivity of $S$.

The other direction, that sections of $\shf{J}_P$ also correspond to sections of $\shf{K}_P$, requires special consideration of the sheaves involved.
First of all, notice that global sections of $\shf{K}_P$ are determined by their value on $I$, 
since this is the stalk over the unique minimal element in the base partial order.
We merely need to show that a global section of $\shf{J}_P$ is also determined by the value on the corresponding element in its base partial order,
even though that element is not the \emph{unique} minimal element.

To this end, suppose that we have a global section of $\shf{J}_P$.
Such a global section consists of the following information:
\begin{itemize}
\item A value $z \in \mathbb{R}^M$, and
\item A set of values $\{w_{n,j}\}_{j=1}^n$ in $I^n/S_n$ for each $n=1, \dotsc, P$, each of which consists of a pair $w_{n,j} = (a_{n,j},b_{n,j})$ because $I$ is a conical space,
\end{itemize}
such that
\begin{equation*}
z = S(w_{n,1}, \dotsc, w_{n,n})
\end{equation*}
for each $n= 1, \dotsc, P$.
Because of the injectivity of $S$, this means that in particular, we have that
\begin{equation*}
z = S(w_{p,1}, \dotsc, w_{p,P}).
\end{equation*}
By the injectivity of $S$, for this particular value of $z$, no other set $\{w_{P,j}\}_{j=1}^P$ will satisfy the equation.
On the other hand, we have that
\begin{equation*}
z = S(w_{1,1}),
\end{equation*}
and again uniquely so.
This means that $w_{1,1}$ must be present in $\{w_{P,j}\}_{j=1}^P$, and furthermore must be the only non-vertex element (recall that the vertex is the equivalence class consisting of all elements with $a_{i,j}=0$).
We merely need to notice that this means that the global section is determined by its value $w_{1,1}$ on $I$.
\end{proof}

The proof of Proposition \ref{prop:k_p_j_p} is easily extended to handle the situation of local sections of $\shf{K}_P$; 
merely trim off the left side of the diagram!

\begin{corollary}
Every local section of $\shf{K}_P$ is also a local section of $\shf{J}_P$.
\end{corollary}

This means that if we use $\shf{K}_P$ instead of $\shf{J}_P$, we essentially have the same information, simply encapsulated in a smaller topology.
The smaller topology might help guide our attempts at finding extensions of assignments with minimal consistency radus, 
and so may be rather convenient even though the consistency radius formula for $\shf{K}_P$ is more complicated than that of $\shf{J}_P$.

\subsection{Determining the number of source adaptively}
\label{sec-source_count}

The general strategy for determining source parameters from measurements is to minimize consistency radius in either the sheaf $\shf{J}_P$ or the sheaf $\shf{K}_P$, subject to the assignment being determined at the observation.  
In cases where the source magnitudes are widely varying, it is likely that a greedy approach is possible (and may be preferable).
Minimizing consistency radius on $\shf{K}_P$ will tend to exhibit the sorting property (Proposition \ref{prop:sorting}), and
therefore will prioritize sources with large magnitude first.
On the other hand, if the source magnitudes are all similar, then a greedy approach is likely to fail when source locations are similar.
In this situation, minimizing consistency radius in $\shf{K}_P$ may yield misleading results; instead minimizing consistency radius on $\shf{J}_P$ would be preferable.

\begin{proposition}
\label{prop:unknown_source_optimization}
Suppose that the number of sources is fixed at $N$ and that $z \in \mathbb{R}^M$ is given by
\begin{equation*}
z= S(a_1,b_1, \dotsc, a_N,b_N),
\end{equation*}
where $S$ is given by Equation \ref{eq:S_function}.

If we assign $z \in \mathbb{R}^M$ to represent the measurement in the sheaf $\shf{J}_P$ (see Equation \ref{eq:j_p}) or $\shf{K}_P$ (see Equation \ref{eq:k_p}) then there is an extension to a global assignment in which the local consistency radius vanishes for all sufficiently small open sets provided that $P \ge N$.
\end{proposition}
\begin{proof}
Establishing the statement for $\shf{K}_P$ automatically carries over to $\shf{J}_P$, so we merely consider the case of $\shf{K}_P$.  That is given by
\begin{equation*}
\xymatrix{
I \ar[rr]\ar[drr]_{S(\cdot)} && I^2/S_2\ar[rr]\ar[d]_{S(\cdot,\cdot)} && I^3/S_3 \ar[rr]\ar[dll]_{S(\cdot,\cdot,\cdot)} && \dotsb \ar[r] & I^P/S_P  \ar[dlllll] \\
&&\mathbb{R}^M\\
}
\end{equation*}
To that end, $z\in \mathbb{R}^M$ is assigned in the bottom row of the diagram.
This is merely a consequence of the structure of the global consistency radius, initially shown in Equation \ref{eq:consistency_radius_fixed_measurements}.
Let us name the values in the assignment to the top row of the sheaf diagram 
\begin{equation*}
\{(a_{1,1},b_{1,1})\}, \{(a_{2,1},b_{2,1}),(a_{2,2},b_{2,2})\}, \{(a_{3,1},b_{3,1}),(a_{3,2},b_{3,2}),(a_{3,3},b_{3,3})\}, \dotsc, \{(a_{P,1},b_{P,1}),\dotsc,(a_{P,P},b_{P,P})\}.
\end{equation*}
Notice that these values define many possibilities for the source parameters.

If the number of sources is not known, but decompositions with $P \ge N$ sources are encoded in the sheaf, 
then the consistency radius of any assignment supported on the star open set (minimal open set in the sense of inclusion) whose stalk is $I^N/S_N$ is given by
\begin{equation}
\label{eq:consistency_radius_fixed_measurements_windowed}
\begin{aligned}
\sum_{i=N}^P \sum_{j=i+1}^P \min_{\sigma \in S_{j}}\left( \sum_{k=1}^{i} d_I\left((a_{i,k},b_{i,k}),(a_{j,\sigma(k)},b_{j,\sigma(k)})\right) + \sum_{m=i+1}^{j}\|a_{j,\sigma(m)}\| \right) +\\ \sum_{i=N}^P \|S(a_{i,1},b_{i,1} \dotsc, a_{i,i},b_{i,i}) - z\|.
\end{aligned}
\end{equation}
(Specifically, in Equation \ref{eq:consistency_radius_fixed_measurements_windowed} rather than starting the $i$ sums at $1$ they start at $N$.  
Additionally, all sums run to $P$ instead of $N$.)

Define
\begin{equation*}
a_{P,k} = \begin{cases}
a_k & \text{if } k \le N,\\
0 & \text{otherwise}
\end{cases}
\end{equation*}
and
\begin{equation*}
b_{P,k} = \begin{cases}
b_k & \text{if } k \le N,\\
\text{arbitrary} & \text{otherwise.}
\end{cases}
\end{equation*}

If we then declare that $a_{i,k} =a_{j,k}$ and $b_{i,k} =b_{j,k}$ whenever they are both defined, then this definition clearly results in the second sum of Equation \ref{eq:consistency_radius_fixed_measurements_windowed} being zero.  
On the other hand, the sum 
\begin{equation*}
\sum_{k=1}^{i} d_I\left((a_{i,k},b_{i,k}), (a_{j,\sigma(k)},b_{j,\sigma(k)})\right)
\end{equation*}
will also be made to vanish, simply by letting $\sigma = \id$.  
What of the remaining sum?  
It too vanishes because all of the terms in question have index greater than $N$, and therefore each vanishes by construction.
\end{proof}

\begin{corollary}
\label{cor:correct_sources}
If the $S$ function is injective, then the assignment for which the local consistency radius of both $\shf{J}_P$ and $\shf{K}_P$ vanishes occurs precisely at the number of sources $N$.
\end{corollary}

Therefore, from a methodological perspective, the number of sources can be determined from the consistency filtration given enough samples.
If there is noise, look for the ``knee'' in local consistency radius as the open set size increases.  
That determines the open set that best describes the data.

As an example, consider the following conical space $I=\mathbb{R}\times\mathbb{R}$ under the metric
\begin{equation*}
d((a_1,b_1),(a_2,b_2)) := \min\{|a_1|+|a_2|, \sqrt{|a_1 - a_2|^2 + \beta|b_1 - b_2|^2}\}.
\end{equation*}
Let us treat $I$ as the representation of a single source, 
whose magnitude and position are each given by a scalar.
We consider two Examples, \ref{eg:two_sources_0} and \ref{eg:two_sources_1} involving the consistency radius of the two source problem.  
In Example \ref{eg:two_sources_0}, both magnitudes are known and have the same value, 
while in Example \ref{eg:two_sources_1} they can have different values.

In both Examples, the sheaf diagram of the problem has the form
\begin{equation*}
\xymatrix{
A_1=I \ar[r]^{i} \ar[dr]_{g} & A_2=I^2/S_2 \ar[d]^f\\
&B=\mathbb{R}^M\\
}
\end{equation*}
where $\mathbb{R}^M$ has the usual Euclidean metric.
This is an example of the sheaf $\shf{K}_2$ in what follows.
To keep the notation clear, we will name the assignment on $A_1$ to be $(a_{11},b_{11})$,
the assignment on $A_2$ to be $\left((a_{21},b_{21}),(a_{22},b_{22})\right)$, 
and the assignment on $B$ to be $z$.
The restriction map $f$ is given by the formula
\begin{equation*}
f\left((a_{21},b_{21}),(a_{22},b_{22})\right) = a_{21} \phi(b_{21}) + a_{22} \phi(b_{22}),
\end{equation*}
where $\phi : \mathbb{R} \to \mathbb{R}^M$ is a smooth function.
Since the top row of the sheaf diagram is an aggregation sheaf, we have that
\begin{equation*}
i((a_{11},b_{11})) = \left((a_{11},b_{11}),(0,c)\right),
\end{equation*}
where $c$ is arbitrary (but not required for well-definedness) since $(0,c)$ is the vertex of $I$. 
Commutativity of the sheaf diagram for $\shf{K}_2$ ensures that the restriction $g = f \circ i$ is given by
\begin{equation*}
g((a_{11},b_{11})) = a_{11} \phi(b_{11}).
\end{equation*}

\begin{example}
\label{eg:two_sources_0}
Let us assume that the magnitudes of both sources have the same value.
Without loss of generality, we may declare that $a_1 = a_2 = 1$.
The $\shf{K}_2$ sheaf model for this two source problem becomes
\begin{equation*}
\xymatrix{
A_1=\mathbb{R} \ar[r] \ar[dr] & A_2=\mathbb{R}^2 \ar[d]^f\\
&B=\mathbb{R}^M\\
}
\end{equation*}
Given the unit magnitudes,
we note that 
\begin{eqnarray*}
d\left(i(1,b_{11}), \left((1,b_{21}),(1,b_{22})\right) \right) &=&d\left(\left((1,b_{11}),(0,c)\right),\left((1,b_{21}),(1,b_{22})\right) \right)^2\\ 
&=& \beta|b_{11}-b_{21}|^2 + \min\{1+0, \sqrt{|0-1|^2 + \beta|c-b_{21}|}\}^2\\
&=& \beta|b_{11}-b_{21}|^2 + 1.
\end{eqnarray*}
This implies that the square of the consistency radius is
\begin{eqnarray*}
c^2(b_{11},b_{21},b_{22}) &=& \beta(b_{11}-b_{21})^2 + 1  + \left(f(b_{21},b_{22})-z\right)^2 + \left(g(b_{11}-z\right)^2\\
&=& \beta(b_{11}-b_{21})^2 + 1  + \left(\phi(b_{21}) + \phi(b_{22})-z\right)^2 + \left(\phi(b_{11})-z\right)^2.
\end{eqnarray*}
An expression of Corollary \ref{cor:correct_sources} is that the third term in the above,
\begin{equation*}
\left(\phi(b_{21}) + \phi(b_{22})-z\right)^2,
\end{equation*}
is the local consistency radius that can vanish if the correct source locations are used, since there are precisely two sources.  
Evidently the global consistency radius cannot vanish (ever)!
\end{example}

\begin{example}
\label{eg:two_sources_1}
Now let us consider the situation in which the source magnitudes are not known.  The sheaf diagram for $\shf{J}_2$ is given by
\begin{equation*}
\xymatrix{
A_1=\mathbb{R}^2 \ar[dr] & A_2=\mathbb{R}^2\times \mathbb{R}^2 \ar[d]^f\\
&B=\mathbb{R}^M\\
}
\end{equation*}
The square of the consistency radius for an assignment to $\shf{J}_2$ is given by
\begin{equation*}
c^2_{\shf{J}_2}(a_{11},b_{11},a_{21},b_{21},a_{22},b_{22}) =\|a_{21}\phi(b_{21}) + a_{22}\phi(b_{22})-z\|^2 + \|a_{11}\phi(b_{11}) - z\|^2.
\end{equation*}
If there are two sources, namely $z = a_{21}\phi(b_{21}) +  a_{22}\phi(b_{22})$, 
minimum consistency radius in $\shf{J}_2$ occurs precisely when the two terms of $c^2_{\shf{J}_2}$ are minimized.
Since this can be done independently, if $a_{22} \not=0$, then the second term will never vanish but the first will vanish.
On the other hand, if $a_{22} = 0$, then the global minimum consistency radius will be zero and will ensure that $(a_{11},b_{11}) = (a_{21},b_{21})$.

Let us now consider what happens with $\shf{K}_2$.
The sheaf diagram for $\shf{K}_2$ is
\begin{equation*}
\xymatrix{
A_1=\mathbb{R}^2 \ar[r] \ar[dr] & A_2=\mathbb{R}^2\times \mathbb{R}^2 \ar[d]^f\\
&B=\mathbb{R}^M\\
}
\end{equation*}
The square of the consistency radius of an assignment to $\shf{K}_2$ contains several more terms than $\shf{J}_2$, and is given by
\begin{eqnarray*}
c^2_{\shf{K}_2}(a_{11},b_{11},a_{21},b_{21},a_{22},b_{22}) &=& d_I\left((a_{11},b_{11}),(a_{21},b_{21})\right)^2 + \|(a_{22},b_{22})\|_I^2 \\&&+ \|a_{21}\phi(b_{21}) + a_{22}\phi(b_{22})-z\|^2 + \|a_{11}\phi(b_{11}) - z\|^2\\
&=& \min\left\{\left(|a_{11}|+|a_{21}|\right)^2,
|a_{11}-a_{21}|^2 + \beta |b_{11}-b_{21}|^2
\right\} + |a_{22}|^2 \\&&+ \|a_{21}\phi(b_{21}) + a_{22}\phi(b_{22})-z\|^2 + \|a_{11}\phi(b_{11}) - z\|^2.
\end{eqnarray*}

Let us assume that the single source problem is a subproblem of the two source problem.
That means that
\begin{equation*}
(a_{11},b_{11})=(a_{21},b_{21}).
\end{equation*}
Under the above assumption, the consistency radius is a function of $a_{21}$, $b_{21}$, $a_{22}$, and $b_{22}$ only,
\begin{equation*}
c^2_{\shf{K}_2}(a_{21},b_{21},a_{22},b_{22}) =
 |a_{22}|^2 + \|a_{21}\phi(b_{21}) + a_{22}\phi(b_{22})-z\|^2 + \|a_{21}\phi(b_{21}) - z\|^2.
\end{equation*}

Assume that the measured signal is from two sources, so that $z = a_{21}\phi(b_{21}) +  a_{22}\phi(b_{22})$. 
The square of the consistency radius becomes
\begin{eqnarray*}
c^2_{\shf{K}_2} &=&
 |a_{22}|^2 + \|a_{21}\phi(b_{21}) - z\|^2 \\
&=& |a_{22}|^2 + \|-a_{22}\phi(b_{22})\|^2 \\
&=& |a_{22}|^2 \left(1 + \|\phi(b_{22})\|^2\right).
\end{eqnarray*}
This vanishes if and only if the measurements are from a single source, since in that case $a_{22}=0$.

On the other hand, where are the critical points?
We begin by a preliminary calculation:
\begin{eqnarray*}
\frac{\partial}{\partial a_{11}} \|a_{11}\phi(b_{11}) - z\|^2 &=& \frac{\partial}{\partial a_{11}} \left(\left(a_{11}\phi(b_{11}) - z\right)\bullet\left(a_{11}\phi(b_{11}) - z\right)\right)\\
&=& 2 \phi(b_{11}) \bullet \left(a_{11}\phi(b_{11}) - z\right).
\end{eqnarray*}
That established, the gradient of the square of the consistency radius is given by
\begin{equation*}
d(c^2_{\shf{K}_2}) = 2 \begin{pmatrix}
\phi(b_{21})\bullet \left(2 a_{21}\phi(b_{21}) + a_{22}\phi(b_{22})-2z\right)\\
\left(d\phi(b_{21})\right)\left(2 a_{21} \phi(b_{21}) + a_{22} \phi(b_{22}) - 2 z\right)\\
a_{22} + \phi(b_{22}) \bullet \left(a_{21}\phi(b_{21}) + a_{22}\phi(b_{22})-z\right)\\
\left(d\phi(b_{22})\right) \left(a_{21}\phi(b_{21}) + a_{22}\phi(b_{22})-z\right)\\
\end{pmatrix}
\end{equation*}

Continuing to assume there are two sources, so that $z = a_{21}\phi(b_{21}) +  a_{22}\phi(b_{22})$,
the gradient of the square of the consistency radius simplifies to
\begin{equation*}
d(c^2_{\shf{K}_2}) = 2 \begin{pmatrix}
\phi(b_{21})\bullet \left(a_{21}\phi(b_{21}) -z\right)\\
\left(d\phi(b_{21})\right)\left(a_{21} \phi(b_{21}) - z\right)\\
a_{22} \\
0\\
\end{pmatrix}.
\end{equation*}
If there is only one source, then each of the coefficients of the gradient do indeed vanish since $a_{22}=0$,
implying that the minimum of the consistency radius is the correct single source decomposition.
On the other hand, if there are two sources, then the gradient cannot vanish in virtue of the third component.
This implies that minimizing consistency radius will only recover approximations to the correct source parameters.
These approximations improve in the limit when $|a_{22}| \ll |a_{21}|$, 
which is effectively a condition required by the greedy algorithm posited by minimizing consistency radius in $\shf{K}_P$.
\end{example}


\section{Application to the two example problems}
\label{sec-application}

This section outlines a few examples of the sheaf constructions described in this article.
It also records the results from several of them,
along with a description of a general framework for sheaf-based calculations.

\subsection{Fourier decomposition}

If the samples are equally spaced, then the signal model is given by
\begin{equation*}
S(a_{-N},-N,a_{-N+1},-N+1, \dotsc, a_{N},N) = \left(\sum_{n=-N}^N a_n, \sum_{n=-N}^N a_n e^{2\pi i n h}, \sum_{n=-N}^N a_n e^{4\pi i n h}, \dotsc, \sum_{n=-N}^N a_n e^{2\pi i n h (M-1)}\right)
\end{equation*}
for some fixed $h \in \mathbb{R}$.

We can express this diagrammatically as 
\begin{equation*}
\xymatrix{
\mathbb{C}^{2N+1} \ar[d]^S \\
\mathbb{C}^M
}
\end{equation*}
noting that minimizing consistency radius is tantamount to inverting the map $S$.
If the number of sources $2N+1$ is known, then $S$ is simply a discrete Fourier transform.
Provided $M=2N+1$, the inverse discrete Fourier transform will recover the source parameters $a_n$ without any further trouble.

\subsection{Spectral estimation}
\label{sec-spectral-solution}

Let us consider the situation described in Section \ref{sec-spectral_estimation},
where the signal model is a sum of complex sinusoids
\begin{equation*}
s(x) = \sum_{n=1}^N a_n e^{2 \pi i \omega_n x},
\end{equation*}
with complex magnitudes $\{a_n\}$ and real frequencies $\{\omega_n\}$ being the unknown parameters.
For each source, the unknown parameters lie in $(\mathbb{C} \times \mathbb{R})$, 
which can be given a conical metric.
If we make $M$ measurements, so that $x \in \mathbb{R}^M$, the $\shf{J}_P$ sheaf therefore has the form
\begin{equation*}
\xymatrix{
(\mathbb{C} \times \mathbb{R}) \ar[drr] & (\mathbb{C} \times \mathbb{R})^2 \ar[dr] & \dotsb & (\mathbb{C} \times \mathbb{R})^P \ar[dl]\\
&& \mathbb{C}^M\\
}
\end{equation*}
According to Proposition \ref{prop:unknown_source_optimization}, this will have zero local consistency radius on sufficiently small open sets provided that $P\ge N$ and the correct source parameters are found.
Therefore, minimization of local consistency radius over all open sets in the Alexandrov topology of the base space of the $\shf{J}_P$ sheaf is guaranteed to recover all source parameters.

The number of open sets in the Alexandrov topology of the base space of $\shf{J}_P$ is $2^P+2$.  Since this count also contains the empty set and the measurement space $\mathbb{C}^M$, this means that at most $2^P$ optimizations need to be performed.  
However, most of the optimization problems are not relevant to the problem at hand -- namely determining the soure parameters.  
In actual fact, only $P$ optimization problems are needed -- one problem for each proposed number of sources.
These problems are to be solved independently, since the sheaf structure does not include restriction maps between the elements in the top row of the diagram.

However, a greedy approach is suggested by the $\shf{K}_P$ sheaf, whose diagram is given by
\begin{equation*}
\xymatrix{
(\mathbb{C} \times \mathbb{R}) \ar[r] \ar[drr] & (\mathbb{C} \times \mathbb{R})^2 \ar[dr] \ar[r] & \dotsb \ar[r] & (\mathbb{C} \times \mathbb{R})^P \ar[dl]\\
&& \mathbb{C}^M\\
}
\end{equation*}
Given the fact that the top row of the $\shf{K}_P$ sheaf is an aggregation sheaf, Proposition \ref{prop:sorting} suggests -- 
but does not establish -- 
that solutions to the minimization of local consistency radius with a small number of sources and minimization of local consistency radius with a larger number of sources ought to be related.
One can therefore proceed somewhat greedily, 
by using the solution for a given number of sources as an initial guess for the solution with either one more or one fewer sources.
From a practical standpoint, both incrementally increasing and incrementally decreasing the number of sources are viable approaches.

In addition to these two sheaves, if measurements are gathered while the optimization is to be performed, 
one may use a different structure that works a little differently.  
Specifically, if $P$ sources are proposed, then since $(\mathbb{C} \times \mathbb{R})^P$ is determined by $3P$ real parameters, 
then the Whitney embedding theorem implies that $3P+1$ complex measurements are needed (which is one real parameter more than necessary).
However, if more sources are proposed, then the measurements already present can simply be reused.
This results in a sheaf of the form
\begin{equation*}
\xymatrix{
(\mathbb{C} \times \mathbb{R})  \ar[d] & (\mathbb{C} \times \mathbb{R})^2 \ar[d]  & \dotsb & (\mathbb{C} \times \mathbb{R})^P \ar[d]\\
\mathbb{C}^4 & \mathbb{C}^7 \ar[l] & \dotsb \ar[l] & \mathbb{C}^{3P+1} \ar[l]\\
}
\end{equation*}
in which the arrows on the bottom row represent projection maps.
It is an immediate consequence of Proposition \ref{prop:unknown_source_optimization} that minimizing local consistency radius in this alternative sheaf will still recover the correct number of sources and their associated source parameters.


\section{Discussion and conclusions}
\label{sec-conclusion}

The topological and sheaf-based approach posited in this article exhibits significant advantages over current approaches to solving modal decomposition problems.
Foremost is that our approach is essentially \emph{the most general possible} approach to solving these problems,
and yet we are able to prove solvability conditions and prescribe algorithmic strategies for constructing solutions.
One cannot overemphasize this fact: regardless of what physical processes or techniques are used for gathering measurements about sources -- be they classical, quantum, or otherwise -- they are subject to the same mathematical formalism as we have discovered and described in this article.
Given the generality of our approach, it is likely that (a) our methods will apply in many unexpected settings, (b) our methods may exactly recover existing techniques in certain settings, but (c) anecdotally (see \cite{robinson2018dynamic} for instance) our more general methods may substantially outperform existing modal decompositions when existing techniques are hampered by limiting assumptions about the signals being used.

We call out three foci where our approach may have a significant impact.
First, our approach can be applied to many other modal decomposition problems without essential change.
Indeed, what needs to be defined are the spaces of source parameters (as conical spaces), the signal map $S$, and the space of possible measurements.

Secondly, because of the generality of the sheaf-based approach to modal decomposition, any other modal decomposition approach can be recast as a sheaf.
It is a mathematical fact that the data within any well-posed data fusion problem can be recast as a sheaf \cite{robinson2017sheaves}.
Recovering the source parameters from measurements is clearly a kind of data fusion problem, although it is one in which the data are of homogeneous type.
Therefore, the sheaf recasting result simply applies to modal decomposition as a special case.  
Although not stated in this way directly, our approach described in Section \ref{sec-approach} is simply a manifestation of the general result in the special case of modal decompositions.

Finally, in recasting a data fusion problem as a sheaf, 
the original algorithmic framework is nearly always transformed into an optimization problem involving consistency radius.  
Even when this does not occur, optimization of consistency radius tends to yield insight into the original approach, and sometimes outperforms the original approach when the original approach's assumptions are violated.
Additionally, anecdotally, the mere task of recasting data fusion problems as sheaves requires up-front engineering of requirements and assumptions.
While this is ``good model hygiene,'' the careful crafting of assumptions is time consuming and is often omitted.  Sheaf-based methods prohibit this, and yet allow one to exploit the performance deficiencies that are created when assumptions are violated.

Since our starting point in Section \ref{sec-problem_statement} was that of the most general modal decomposition, 
our theoretical approach is in the privileged position of acting as an unbiased arbitrator amongst algorithms.  
Indeed, since any other approach to solving a modal decomposition problem can be recast as a sheaf model,
this gives us the ability to compare several such sheaf models in an even-handed way.
Using our sheaf-based framework as a baseline since other approaches may be recast as sheaves, apparently distinct approaches are able to be compared in an unbiased manner.

\section*{Acknowledgements}
The views, opinions and/or findings expressed are those of the author and should not be interpreted as representing the official views or policies of the Department of Defense or the U.S. Government.
This material is based upon work supported by the
Defense Advanced Research Projects Agency (DARPA) under Agreement No. HR00112090125.


\bibliographystyle{plain}
\bibliography{aggsheaf-master}

\end{document}